\documentclass[submission,copyright,creativecommons]{eptcs}
\providecommand{\event}{GASCom 2024} 

\usepackage{iftex}

\ifpdf
  \usepackage{underscore}         
  \usepackage[T1]{fontenc}        
\else
  \usepackage{breakurl}           
\fi

\usepackage{amsmath,amsthm,amssymb}

\usepackage{hyperref}
\usepackage{comment}
\usepackage{paralist}
\usepackage{thmtools, thm-restate}

\newtheorem{thm}{Theorem}[section]
\newtheorem{prop}[thm]{Proposition}
\newtheorem{cor}[thm]{Corollary}
\newtheorem{lem}[thm]{Lemma}
\newtheorem{obs}[thm]{Observation}

\newtheorem{prob}[thm]{Problem}

\usepackage{paralist}
\usepackage{lipsum}
\usepackage{thm-restate}

\usepackage{enumerate}
\usepackage{tikz}
\usepackage{tabularx}
\usepackage{mathdots}

\numberwithin{equation}{section}

\newcommand{\sn}{\mathfrak{S}_n}

\newcommand{\zsn}{\mathbb{Z}[\sn]}

\newcommand{\zqq}{\mathbb{Z}[\qp12, \qm12]}

\newcommand{\qp}[2]{q^{\frac{#1}{#2}}}
\newcommand{\qm}[2]{q^{\negthinspace\Bar\,\frac{#1}{#2}}}

\newcommand{\qdiff}{\qp12 - \qm12}

\newcommand{\wT}{\widetilde T}

\newcommand{\hnq}{H_n(q)}

\newcommand{\anq}{\mathcal{A}_n(q)}

\newcommand{\wtc}[2]{\widetilde{C}_{#1}(#2)}

\newcommand{\imm}[1]{\mathrm{Imm}_{#1}}

\newcommand{\sumsb}[1]{\sum_{\substack{#1}}}  

\newcommand{\pinv}
{\textsc{inv}_{\ntnsp P}}

\newcommand{\defeq}{:=} 

\newcommand{\avoidsp}{avoids the patterns $3412$ and $4231${}}

\newcommand{\avoidingp}{avoiding the patterns $3412$ and $4231${}}

\newcommand{\tr}{{\negthickspace \top \negthickspace}}

\newcommand{\ntnsp}{\negthinspace}

\newcommand{\nTksp}{\negthickspace\negthickspace}

\newcommand{\bp}{\begin{prob}}
\newcommand{\ep}{\end{prob}}

\newcommand{\inc}{\mathrm{inc}}

\newcommand{\trspace}[1]{\mathcal{T}(#1)}
\newcommand{\upparrow}{\big \uparrow \nTksp \phantom{\uparrow}}

\newcommand{\asc}{\mathrm{asc}}

\newcommand{\LLT}{\mathrm{LLT}}
\newcommand{\llt}[1]{\mathrm{LLT}_{#1,q}}

\usepackage{graphicx}

\title{LLT Polynomials and Hecke Algebra Traces}
\author{Alejandro H. Morales\thanks{AHM is partially supported by NSF grant DMS-2154019.}
\institute{LACIM \\Universit\'e du Qu\'ebec \`a Montr\'eal \\Montr\'eal QC, Canada}
\email{\href{mailto:morales_borrero.alejandro@uqam.ca}{morales\_borrero.alejandro@uqam.ca}}
\and
Mark A. Skandera \qquad\qquad Jiayuan Wang\thanks{JW is partially supported by an AMS-Simons travel grant.}
\institute{Department of Mathematics\\Lehigh University\\
Bethlehem PA, USA}
\email{\quad \href{mailto:mas906@lehigh.edu}{mas906@lehigh.edu} \quad\qquad \href{mailto:jiw922@lehigh.edu}{jiw922@lehigh.edu}}
}
\def\titlerunning{LLT Polynomials and Hecke Algebra Traces}
\def\authorrunning{A. H. Morales, M. A. Skandera \& J. Wang}
\begin{document}
\maketitle

\begin{abstract}
We show that coefficients in unicellular LLT polynomials are evaluations of Hecke algebra traces at Kazhdan--Lusztig basis elements. We express these in terms of traditional trace bases, induction, and Kazhdan--Lusztig $R$-polynomials.
\end{abstract}

\section{Introduction}

The study of proper colorings of
a graph $G$ is a fundamental topic in discrete mathematics. Stanley \cite{StanSymm} defined the chromatic symmetric function $X_{G,q}$ which is a symmetric function generalization of the chromatic polynomial. This function was generalized by Shareshian and Wachs \cite{SWachsChromQF}: for a graph $G=(V,E)$ let 
\(
X_{G,q} := \sum_{ \kappa} q^{\asc(\kappa)} x_{\kappa(1)} x_{\kappa(2)}\cdots,
\) 
where the sum is over all proper
colorings $\kappa:V(G)\to \mathbb{N}$ of $G$ and $\asc(\kappa)$ denotes the number of ascents of $\kappa$, pairs $(i,j)$ with $i<j$ such that $\kappa(i)<\kappa(j)$.
When $G=\inc(P)$ is an incomparability graph of a unit interval order $P$, $X_{G,q}$ is a symmetric function. There are important positivity conjectures about $X_{\inc(P),q}$ like the $e$-positivity conjecture of Stanley--Stembridge--Shareshian--Wachs.

In another context, the functions $X_{\inc(P),q}$ appeared in the
study of the space of {\em diagonal harmonics}. Let $\LLT_{\inc(P),q} := \sum_{ \kappa} q^{\textup{asc}(\kappa)} x_{\kappa(1)}
x_{\kappa(2)}\cdots$, where the sum is over arbitrary vertex
colorings of $\inc(P)$. This is also a symmetric function called a 
{\em unicellular LLT polynomial}, a special case of a family of symmetric functions introduced by
Lascoux--Leclerc--Thibon in 1997 in a different context. These functions
$\LLT_{\inc(P),q}$ appear in the {\em Shuffle conjecture} of diagonal
harmonics \cite{HHLRU05} proved by
Carlsson--Mellit \cite{CM}. In their proof of the shuffle conjecture, they show that both these symmetric functions are related by a {\em plethystic substitution}:
\[X_{\inc(P),q}[X] = (q-1)^{-n} \LLT_{\inc(P),q}[(q-1)X],\]
where $n$ is the size of $P$. From work of Grojnowski and Haiman  \cite{GroHai}, $\LLT_{\inc(P)}$ are Schur positive and it is an open question to find a combinatorial interpretation for this expansion.

An important basis of the Hecke algebra $H_n(q)$ is the (modified, signless) {\em Kazhdan--Lusztig basis} defined by 
\(
\wtc wq := q^{\frac{\ell(w)}{2}} C'_w(q) = \sum_{u\leq v} P_{u,w}(q) T_w,
\)
where $\{T_w \mid w \in \mathfrak{S}_n\}$ is the natural basis of $H_n(q)$, $P_{v,w}(q)$ are the {\em Kazhdan--Lusztig polynomials} and $\leq$ denotes the Bruhat order of $\mathfrak{S}_n$.
It is known \cite{CHSSkanEKL} that the various expansions  of the chromatic quasi-symmetric function $X_{\inc(P),q}$ can be viewed as evaluations of traces at $\{\widetilde{C}_w(q) \mid w \text{ avoiding } 312\}$,
when $P = P(w)$ is a unit interval order corresponding to $w$,
\begin{equation*}
    X_{\inc(P(w)),q} 
    = \sum_{\lambda \vdash n} \epsilon_q^\lambda(\wtc wq) m_\lambda 
    = \sum_{\lambda \vdash n} \eta_q^{\lambda}(\wtc wq) f_\lambda 
    = \sum_{\lambda \vdash n} \chi_q^{\lambda^\tr}(\wtc wq) s_\lambda 
    = \cdots,
\end{equation*}
where $\epsilon_q^\lambda$, $\eta_q^\lambda$, $\chi_q^{\lambda^\tr}$ are induced sign, induced trivial, and irreducible characters of $\hnq$. In this context, the $e$-positivity conjecture of $X_{\inc(P),q}$ is part of a more general conjecture of Haiman \cite{HaimanHecke} for symmetric functions associated to $\wtc wq$ for any $w$ in the context of immanants. In \cite[\S 11.3]{abreu2022geometric}, Abreu and Nigro used the plethystic relation to define analogs of unicellular LLT polynomials for all permutations. 

The coefficients of various expansions of LLT polynomials can also be viewed as evaluations of traces at $\{\widetilde{C}_w(q) \mid w \text{ avoiding } 312\}$. 
\[
\LLT_{\inc(P),q} = \sum_{\lambda} \epsilon^{\lambda}_{q,LLT}(\widetilde{C}_w(q)) m_{\lambda},
\]
where $\epsilon^{\lambda}_{q,LLT}(\widetilde{C}_w(q))$ is a certain LLT-analog of the trace $\epsilon^{\lambda}_q$. 

We describe similar analogs of  induced trivial characters $\eta^{\lambda}_{q,\LLT}$
and power sum traces $\psi^{\lambda}_{q,\LLT}$. The evaluations at $\widetilde{C}_w(q)$ were known as expansions of the LLT but now we obtain evaluations at the natural basis $T_w$, which were not known before.

We also give change of basis equations between $\psi^{n}_{q,\LLT}$, $\epsilon^n_{q,\LLT}$, and $\eta^n_{q,\LLT}$ and known traces that resemble the Cauchy identity of symmetric functions after a principal specialization.

\section{Background}\label{sec: background}

\noindent{\bf Symmetric functions.} We mostly use the notation from \cite[Ch. 7]{StanEC2}. We denote by $\Lambda_n$ the ring of symmetric functions of degree $n$ and $m_{\lambda}$, $e_{\lambda}$, $h_{\lambda}$, $p_{\lambda}$, $s_{\lambda}$, $f_{\lambda}$ denote the {\em monomial}, {\em elementary}, {\em complete}, {\em power sum}, {\em Schur}, and {\em forgotten} symmetric functions. Also $\omega$ denotes the standard involution in $\Lambda_n$.   

\noindent{\bf Hecke algebra and traces.} The Hecke algebra $\hnq$ is a noncommutative
$\zqq$-algebra
generated by {\em natural generators}
$\{ T_{s_i} \,|\, 1 \leq i \leq n-1 \}$ 
subject to the relations
\begin{alignat*}{2}
T_{s_i}^2 &= (q-1) T_{s_i} + q,          
&\qquad &\text{for } i=1,\dotsc,n-1, \\
T_{s_i} T_{s_j} T_{s_i} &= T_{s_j} T_{s_i} T_{s_j}, 
&\qquad &\text{if }  |i-j|=1,\\
T_{s_i} T_{s_j} &= T_{s_j} T_{s_i},         
&\qquad &\text{if }  |i-j| \geq 2.
\end{alignat*}
Specializing $\hnq$ at $\qp 12 = 1$, we obtain the classical group algebra $\zsn$ of the symmetric group.

Let $\trspace{\hnq}$ be the $\zqq$-module of
$\hnq$-{\em traces}, linear functionals
$\theta_q: \hnq \rightarrow \zqq$ satisfying
$\theta_q(DD') = \theta_q(D'D)$ for all $D, D' \in \hnq$.
For any trace $\theta_q: T_w \mapsto a(q)$ in $\trspace{\hnq}$,
the $\qp12 =1$ specialization $\theta: w \mapsto a(1)$ belongs to the space $\trspace{\sn} \defeq \trspace{H_n(1)}$ of $\zsn$-traces
from $\zsn \rightarrow \mathbb Z$ ($\sn$-class functions). Like the $\mathbb Z$-module $\Lambda_n$ of homogeneous degree-$n$ symmetric functions, the trace spaces $\trspace{\hnq}$ and $\trspace{\sn}$ have
dimension equal to the number of {\em integer partitions} of $n$,
the weakly decreasing positive integer sequences
$\lambda = (\lambda_1,\dotsc,\lambda_r)$ satisfying
$\lambda_1 + \cdots + \lambda_r = n$.

It can be useful to record trace evaluations
in a symmetric generating function. In particular, for $D \in \mathbb Q(q) \otimes \hnq$,
we record induced sign character evaluations by defining
\begin{equation}\label{eq:ygdef}
  Y_q(D) \defeq \sum_{\lambda \vdash n} \epsilon_q^\lambda(D) m_\lambda
  \in \mathbb{Q}(q) \otimes \Lambda_n.
\end{equation}
This symmetric generating function in fact gives us all of the standard trace evaluations.

\begin{prop}\label{p:Yexpand}
  The symmetric function $Y_q(D)$ is equal to
  \begin{equation*}\label{eq:sumeta}
      \sum_{\lambda \vdash n} \eta_q^\lambda(D)f_\lambda
    = \sum_{\lambda \vdash n}\frac{sgn(\lambda)\psi_q^\lambda(D)}{z_\lambda}p_\lambda
    = \sum_{\lambda \vdash n} \chi_q^{\lambda^\tr}(D)s_\lambda
    = \sum_{\lambda \vdash n} \phi_q^\lambda(D)e_\lambda
    = \sum_{\lambda \vdash n} \gamma_q^\lambda(D)h_\lambda,
  \end{equation*}
where $sgn(\lambda):=(-1)^{n-\ell(\lambda)}$; equivalently, $\omega Y_q(D)$ is equal to  
\begin{equation*}\label{eq:sumepsilon}
    \sum_{\lambda \vdash n} \epsilon_q^\lambda(D)f_\lambda
  = \sum_{\lambda \vdash n} \eta_q^\lambda(D)m_\lambda
  = \sum_{\lambda \vdash n}\frac{\psi_q^\lambda(D)}{z_\lambda}p_\lambda
  = \sum_{\lambda \vdash n} \chi_q^\lambda(D)s_\lambda
  = \sum_{\lambda \vdash n} \phi_q^\lambda(D)h_\lambda
  = \sum_{\lambda \vdash n} \ntnsp \gamma_q^\lambda(D)e_\lambda.\nTksp
\end{equation*}
\end{prop}

\noindent {\bf Quantum matrix bialgebra and immanants.}
An important computational tool in the evaluation of $\hnq$-traces is the {\em quantum matrix bialgebra} $\anq$, the noncommutative
ring generated as a $\zqq$-algebra by the $n^2$ variables
$t = (t_{i,j})_{i,j \in [n]}$ subject to relations
\begin{equation}\label{eq:qringrelations}
\begin{alignedat}{2}
t_{i,\ell}t_{i,k} &= \smash{\qp12}t_{i,k} t_{i,\ell}, &\qquad
t_{j,k} t_{i,\ell} &= t_{i,\ell}t_{j,k}\\
t_{j,k}t_{i,k} &= \smash{\qp12}t_{i,k} t_{j,k} &\qquad
t_{j,\ell} t_{i,k} &= t_{i,k} t_{j,\ell} + \smash{(\qdiff)} t_{i,\ell}t_{j,k},
\end{alignedat}
\end{equation}
for all indices $1 \leq i < j \leq n$ and $1 \leq k < \ell \leq n$.
As a $\smash{\zqq}$-module, $\anq$ has a natural basis of monomials 
$t_{\ell_1,m_1} \cdots t_{\ell_r,m_r}$ in which index pairs
appear in lexicographic order. The relations (\ref{eq:qringrelations}) 
allow one to express other monomials in terms of this natural basis.

To state immanant generating functions for
$\hnq$-traces,
it will be convenient to express monomials in $\anq$ as follows.
Given $u = u_1 \cdots u_n$, $v = v_1 \cdots v_n \in \sn$, define
\begin{equation*}
  t^{u,v} \defeq t_{u_1,v_1} \cdots t_{u_n,v_n}.
\end{equation*}
For any linear function $\theta_q: \hnq \rightarrow \zqq$,
define the $\theta_q$-immanant in $\anq$ to be
\begin{equation*}
  \imm{\theta_q}(t) 
  = \sum_{w \in \sn} \qm{\ell(w)}2 \theta_q(T_w) t^{e,w}.
\end{equation*}

\begin{prop} \label{conj:product traces}
  Given Hecke algebra traces
  \begin{equation*}
    \theta_1 \in \trspace{H_k(q)},\quad
    \theta_2 \in \trspace{H_{n-k}(q)},\quad
    \theta = (\theta_1 \otimes \theta_2) \upparrow_{H_k(q) \times H_{n-k}(q)}^{\hnq}
    \in \trspace{\hnq},
  \end{equation*}
  we have
  \begin{equation*}
    \imm{\theta}(t) = \sumsb{I\text{ where }|I|=k}\,
    \imm{\theta_1}(t_{I,I}) \imm{\theta_2}(t_{\overline{I},\overline{I}}).
  \end{equation*}
\end{prop}

Since Hecke algebra traces are determined by the values on minimum length representatives (see \cite[Cor. 8.2.6]{GPCharHecke}), then the following result will be useful.

\begin{lem} \label{lem: min reps are smooth}
    Let $w \in \sn$ be of minimum length in its conjugacy class, then 
    $w$ \avoidsp. 
    Furthermore, each 
    $v\leq w$ also \avoidsp,
    and also is of minimum length in its conjugacy class.
\end{lem}

\section{Plethystically defined characters}\label{sec: plethystic characters}

Suppose that a certain plethystic substitution transforms
symmetric functions written $\{ Y_q(D) \,|\, D \in \hnq \}$
into symmetric functions $\{ Z_q(D) \,|\, D \in \hnq \}$, i.e. 
\begin{equation}\label{eq:ratlpleth}
Z_q(D) \defeq r(q) Y_q(D)[s(q)X]
\end{equation}
for some rational functions $r(q)$ and $s(q)$.
This substitution yields a transformation of $\hnq$ traces
\(
  \theta_q \mapsto \theta_{q,Z}
\)
as well.  Namely, we define $\epsilon_{q,Z}^\lambda$
to be the $\hnq$-trace that maps $D$ to the coefficient of $m_\lambda$
in the monomial expansion of $Z_q(D)$,
\begin{equation} \label{eq: definition epsilonqLLT}
  Z_q(D) = \sum_{\lambda \vdash n} \epsilon_{q,Z}^\lambda(D) m_\lambda.
\end{equation}
Then we extend linearly over $\zqq$, mapping
\begin{equation}\label{eq:thetaqz}
  \theta_q = \sum_{\lambda \vdash n} b_\lambda \epsilon_q^\lambda
  \quad \mapsto \quad
  \theta_{q,Z} \defeq \sum_{\lambda \vdash n} b_\lambda \epsilon_{q,Z}^\lambda.
\end{equation}

\begin{obs} \label{obs: LLT char different basis}
     The symmetric function $Z_q(D)$ is equal to
  \begin{equation*}
  \begin{aligned}
      \sum_{\lambda \vdash n} \eta_{q,Z}^\lambda(D)f_\lambda
    &= \sum_{\lambda \vdash n}\frac{sgn(\lambda)\psi_{q,Z}^\lambda(D)}{z_\lambda}p_\lambda
    = \sum_{\lambda \vdash n} \chi_{q,Z}^{\lambda^\tr}(D)s_\lambda= \sum_{\lambda \vdash n} \phi_{q,Z}^\lambda(D)e_\lambda
    = \sum_{\lambda \vdash n} \gamma_{q,Z}^\lambda(D)h_\lambda.
    \end{aligned}
  \end{equation*}
\end{obs}
\begin{prop}
For any plethystically defined map $Y \mapsto Z$
(\ref{eq:ratlpleth})
of symmetric functions,
if $\theta_q$ is a trace function, then so is $\theta_{q,Z}$.
\end{prop}

Furthermore by \eqref{eq:thetaqz}, 
the change of basis matrix which relates two symmetric function bases (and necessarily the traces which correspond by the Frobenius map), also relates the $Z$-analogs of those traces.

For example when $Z = \llt{\inc(P)}$ and $w = w(P)$ \avoidingp, we have
\begin{equation} \label{eq:epsilon for smooth}
\epsilon_{q,\LLT}^n(\wtc wq)=1.
\end{equation}
This is because by Prop.~\ref{prop: monomial expansion LLT} there is only one column-strict Young tableau $U$ of shape $1^n$, the tableau consisting one column with entries $1,2,\ldots,n$ in order,  with $\pinv(U)=0$ where $P=P(w)$. 

It is possible to
describe LLT analogs of 
induced sign characters, induced trivial characters, and power sum traces in terms of character induction.

\medskip 

\noindent {\bf LLT analogs of power sum traces} The LLT analogs of the power sum trace can be expressed simply in terms of the ordinary power sum trace. 
\begin{prop} \label{prop:relating power sum char and LLT power sum}
We have
\[
\psi_{q,\LLT}^{\lambda} = (q-1)^n \prod_{i} \frac{1}{q^{\lambda_i}-1} \cdot \psi_q^{\lambda}.
\]
\end{prop}

\begin{proof}
Omitted.
\end{proof}

\noindent {\bf LLT analogs of induced sign characters and induced trivial characters}
The evaluations of $\epsilon_{q,\LLT}^\lambda$ and $\eta_{q,\LLT}^{\lambda}$ at $\{\widetilde{C}_w(q) \mid w \text{ avoiding } 312\}$ have simple combinatorial interpretations.
\begin{prop} \label{prop: monomial expansion LLT}
  Fix $w \in \sn$ avoiding the pattern $312$, and let $P = P(w)$.  For all $\lambda \vdash n$
  we have
  \begin{equation*}
    \epsilon_{q,\LLT}^\lambda (\wtc wq) = 
    \sum_U q^{\pinv(U)},
  \end{equation*}
  where the sum is over all column-strict Young tableaux $U$
  of shape $\lambda^\tr$, and 
  \begin{equation*}
    \eta_{q,\LLT}^\lambda(\wtc wq) = \sum_U q^{\pinv((U_1 \circ \cdots \circ U_r)^R)},
  \end{equation*}
  where the sum is over all row-strict Young tableaux $U$ 
  of shape $\lambda$ and $(U_1 \circ \cdots \circ U_r)^R$ is the reversal of the concatenation of rows in $U$.
\end{prop}

\begin{proof}
Omitted.
\end{proof}

 This leads to the following generating functions for $\epsilon_{q,\LLT}^{\lambda}$ and $\eta_{q,\LLT}^{\lambda}$.
\begin{thm}\label{t:epsilonllt}
  For $\lambda = (\lambda_1,\dotsc,\lambda_r) \vdash n$ we have
  \begin{equation}\label{eq:epsilonllt}
    \imm{\epsilon_{q,\LLT}^{\lambda}}(t) =
    \sum_{(I_1,\dotsc,I_r)} (t_{I_1,I_1})^{e,e} \cdots (t_{I_r,I_r})^{e,e},
  \end{equation}
  where the sum is over all ordered set partitions $(I_1,\ldots,I_r)$ of type $\lambda$,
  and $e$ is the identity permutation in the appropriate subgroup of $\sn$. And 

  \begin{equation}\label{eq:etallt}
    \imm{\eta_{q,\LLT}^{\lambda}}(t) = \sum_{(I_1,\dotsc,I_r)}
    (t_{I_r,I_r})^{w_0,w_0} \cdots (t_{I_1,I_1})^{w_0,w_0},
  \end{equation}
  where the sum is over all ordered set partitions of type $\lambda$
  and $w_0$ is the longest permutation in the appropriate subgroup of $\sn$.
\end{thm}

\begin{proof}
Omitted.
\end{proof}

Equivalently, we have the following.
\begin{thm}
\label{thm: induced sign LLT}
  We have
\begin{equation*}
    \epsilon_{q,\LLT}^\lambda = 
    (\epsilon_{q,\LLT}^{\lambda_1} \otimes \cdots \otimes \epsilon_{q,\LLT}^{\lambda_r})
    \upparrow_{H_{\lambda}(q)}^{\hnq}, \quad \text{where} \quad \epsilon^n_{q,\LLT}( T_w) = \begin{cases}
      1 & \text{if $w = e$},\\
      0 & \text{otherwise}.
    \end{cases}
  \end{equation*}

  And \begin{equation*}
    \eta_{q,\LLT}^\lambda = 
    (\eta_{q,\LLT}^{\lambda_1} \otimes \cdots \otimes \eta_{q,\LLT}^{\lambda_r})
    \upparrow_{H_{\lambda}(q)}^{\hnq}, \quad \text{where} \quad  \eta^n_{q,\LLT}( T_w) = 
    R_{e,w}(q).
  \end{equation*}
  \end{thm}

\begin{proof}
By Theorem~\ref{t:epsilonllt} and Proposition~\ref{conj:product traces}, we have $\imm{\epsilon_{q,\LLT}^n}(t) = t_{1,1} \cdots t_{n,n}$.

Also, we have $t^{w_0,w_0}= \sum\limits_{w\in \sn} R_{e,w}(q)q^{-\frac{\ell(w)}{2}}t^{e,w}$, and $\eta_{q,\LLT}^{n}(\wT_w) = R_{e,w}(q)q^{-\frac{\ell(w)}{2}}$.
\end{proof}

We can express $\epsilon_{q,\LLT}^n$ and $\eta_{q,\LLT}^n$ in terms of ordinary $\hnq$-characters and principal specialization of symmetric functions.
\begin{cor} \label{cor: power sum LLT in other bases}
\begin{multline*}
\frac{\epsilon_{q,\LLT}^n}{(1-q)^n}  
\,=\, \sum_{\lambda} \frac{1}{z_{\lambda}} \prod_i \frac{1}{1-q^{\lambda_i}} \psi_q^{\lambda}
= \sum_{\lambda} \frac{q^{b(\lambda)}}{\prod_{u\in \lambda} (1-q^{h(u)})} \chi_q^{\lambda} \\ \,=\,  \sum_{\lambda} \prod_{i} \frac{1}{(1-q)(1-q^2)\cdots (1-q^{\lambda_i})} \phi_q^{\lambda}
=  \sum_{\lambda} \prod_{i} \frac{q^{\binom{\lambda_i}{2}}}{(1-q)(1-q^2)\cdots (1-q^{\lambda_i})} \gamma_q^{\lambda},
\end{multline*}
and 
\begin{multline*}
\frac{\eta_{q,\LLT}^n}{(1-q)^n} \,=\, \sum_{\lambda} \frac{\mathrm{sgn}(\lambda)}{z_{\lambda}} \prod_i \frac{1}{1-q^{\lambda_i}} \psi_q^{\lambda}\label{eq: eta LLT in terms of power sum traces} 
\,=\,  \sum_{\lambda} \frac{q^{b(\lambda')}}{\prod_{u\in \lambda'} (1-q^{h(u)})} \chi_q^{\lambda},\\
\,=\,  \sum_{\lambda} \prod_{i} \frac{q^{\binom{\lambda_i}{2}}}{(1-q)(1-q^2)\cdots (1-q^{\lambda_i})}\phi_q^{\lambda}
\,=\,    \sum_{\lambda} \prod_{i}  \frac{1}{(1-q)(1-q^2)\cdots (1-q^{\lambda_i})} \gamma_q^{\lambda},
\end{multline*}   
where $h(u)$ is the hook-length of $u$ and $b(\lambda)=\sum_i (i-1)\cdot \lambda_i = \sum_i \binom{\lambda'_i}{2}$.
\end{cor}

\begin{proof}
Omitted.
\end{proof}

\nocite{*}
\bibliographystyle{eptcs}
\bibliography{my}

\end{document}